\documentclass[11pt, a4paper,pdftex]{amsart} 
\usepackage[foot]{amsaddr}
\author{Jeffrey Giansiracusa$^1$} 
\email{jeffrey.giansiracusa@durham.ac.uk} 

\author{Noah  Giansiracusa$^2$} 
\email{ngiansiracusa@bentley.edu}

\usepackage{mathptmx}
\DeclareSymbolFont{cmlargesymbols}{OMX}{cmex}{m}{n}
\DeclareMathSymbol{\mycoprod}{\mathop}{cmlargesymbols}{"60}

\usepackage[mathcal]{eucal}
\DeclareFontFamily{OT1}{pzc}{}
\DeclareFontShape{OT1}{pzc}{m}{it}{<-> s * [1.10] pzcmi7t}{}
\DeclareMathAlphabet{\mathpzc}{OT1}{pzc}{m}{it}

\usepackage[protrusion=true,expansion=true]{microtype}

\usepackage{amssymb}
\usepackage{mathrsfs}  
\usepackage{latexsym}
\usepackage{amsmath}
\usepackage{color}
\usepackage{hyperref}

\usepackage[cmtip,arrow]{xy}
\usepackage{pb-diagram,pb-xy}
\usepackage[vmargin=3cm, hmargin=3cm]{geometry}
\numberwithin{equation}{subsection}

\newtheorem{thmA}{Theorem}

\newtheorem{theorem}{Theorem}[subsection]  
\newtheorem{lemma}[theorem]{Lemma} 
\newtheorem{proposition}[theorem]{Proposition}
\newtheorem{corollary}[theorem]{Corollary}

\theoremstyle{remark} 
\newtheorem{definition}[theorem]{Definition}
\newtheorem{question}[theorem]{Question}
\newtheorem{remark}[theorem]{Remark}
\newtheorem{example}[theorem]{Example}

\newcommand{\R}{\mathbb{R}}
\newcommand{\N}{\mathbb{N}}
\newcommand{\Z}{\mathbb{Z}}

\newcommand{\T}{\mathbb{T}}

\newcommand{\Fun}{{\mathbb{F}_1}}

\newcommand{\spec}{\mathrm{Spec}\:}

\newcommand{\Trop}{\mathpzc{Trop}}
\newcommand{\trop}{\mathpzc{trop}}
\newcommand{\val}{\mathpzc{Val}}
\newcommand{\bend}{\mathpzc{B}}

\newcommand{\supp}{\mathrm{supp}}
\newcommand{\Sch}{\mathrm{Sch}}

\newcommand{\an}{\mathrm{an}}
\newcommand{\Lan}{\mathrm{Lan}}

\newcommand{\aff}{\mathit{aff}}
\newcommand{\op}{\mathrm{op}}
\newcommand{\Frac}{\mathrm{Frac}}

\definecolor{purple}{rgb}{0.5,0,0.5}
\definecolor{brown}{rgb}{0.5,0.3,0.1}

\DeclareMathOperator*{\colim}{colim} 

\mathchardef\mhyphen="2D

\title{The universal tropicalization and the Berkovich analytification} 

\date{9 August, 2022}

\begin{document}
\begin{abstract}
  Given an integral scheme $X$ over a non-archimedean valued field $k$, we construct a
  universal closed embedding of $X$ into a $k$-scheme equipped with a model over the field with one
  element $\Fun$ (a generalization of a toric variety).  An embedding into such an ambient space
  determines a tropicalization of $X$ by \cite{GG1}, and we show that the set-theoretic
  tropicalization of $X$ with respect to this universal embedding is the Berkovich analytification
  $X^{\an}$.  Moreover, using the scheme-theoretic tropicalization of \cite{GG1}, we obtain a
  tropical scheme $\Trop_{univ}(X)$ whose $\T$-points give the analytification and which canonically
  maps to all other scheme-theoretic tropicalizations of $X$.  This makes precise the idea that the
  Berkovich analytification is the universal tropicalization.  When $X=\spec A$ is affine, we show
  that $\Trop_{univ}(X)$ is the limit of the tropicalizations of $X$ with respect to all embeddings
  in affine space, thus giving a scheme-theoretic enrichment of a well-known result of Payne.
  Finally, we show that $\Trop_{univ}(X)$ represents the moduli functor of semivaluations on $X$,
  and when $X=\spec A$ is affine there is a universal semivaluation on $A$ taking values in the
  idempotent semiring of regular functions on the universal tropicalization.
\end{abstract}
\maketitle

\section{Introduction}
In recent years, two methods of translating problems from algebraic geometry into other landscapes,
both based on non-archimedean valuations, have become increasingly important.  The first is
\emph{tropicalization} \cite{Kapranov, MS,Mikhalkin-ICM}, which reduces the complexity of varieties
by turning them into finite polyhedral complexes.  The second is Berkovich's non-archimedean
\emph{analytification} \cite{Berkovich}, where even the affine line becomes an intricate
fractal-like infinitely branching tree.

Analytification is intrinsic, whereas tropicalization depends on the choice of an embedding into a
toric variety.  Payne showed that these two processes are intimately related: the Berkovich
analytification of an affine variety over a complete valued field is the category-theoretic limit of
its tropicalizations with respect to embeddings in affine spaces (see \cite[Theorem 1.1]{Payne}, as
well as \cite[Theorem 4.2]{Payne}, \cite{Payne-inv2}, and \cite[Theorem 6.4]{Ulirsch-divisorial} for
global variants).  Therefore the analytification is often thought of as an intrinsic and universal
tropicalization.  This view is also hinted at in the unfinished paper \cite{Kont-Tschink}, and
further justified by the existence of skeletons --- polyhedral complexes resembling tropical
varieties onto which the analytification admits a strong deformation retraction
\cite{Berkovich,Berkovich-locally-contractible-I,Berkovich-locally-contractible-II,Hrushovski-Loeser-tame,Thuillier,Gubler-skeleton,Kontsevich-Soibelman, Nicaise-Mustata}.

A second way to think of the analytification is as a space of
semivaluations.  Given a non-archimedean valued field $k$ and a $k$-algebra $A$, the Berkovich
analytification of $\spec A$ can be defined as the space of all rank-one semivaluations on $A$
compatible with the valuation on $k$, and this description extends to non-affine schemes with an
appropriate notion of (semi)valuation in the non-affine case.

Thus we have two heuristic devices with which to view the analytification: as the universal
tropicalization, and as the moduli space of semivaluations.  Using the ideas developed in \cite{GG1} --- (1) a
generalization of the ambient spaces in which to tropicalize from toric varieties to schemes
equipped with models over the field with one element, (2) a refinement of tropicalization from sets to
semiring schemes, and (3) a generalization of semivaluations to take values in arbitrary idempotent
semirings --- we are now able to make both of these ideas into precise statements.

\subsection{The universal tropicalization}

The embeddings originally used for tropicalization were taken to be in an algebraic torus, as in
\cite{Sturmfels-first-steps, Kapranov}, but Payne and Kajiwara showed that there is a natural
extension to arbitrary toric varieties simply by computing the tropicalization separately for each
torus-invariant stratum and assembling the results \cite{Payne, Kajiwara}.  The underlying reason that
toric varieties yield tropicalizations of their subschemes is that they have a distinguished class
of monomials in the coordinate ring of each torus-invariant affine patch, and all the gluing of
affine patches occurs by localization of monomials.  

One is thus led immediately to consider a more general class of ambient spaces in which to
tropicalize subschemes: $k$-schemes $Z$ equipped with a model over $\Fun$, the field with one
element.  This means an $\Fun$-scheme $Z'$ (essentially in the sense of \cite{Toen-Vaquie} or
\cite{Kato, Deitmar1, Deitmar2}) and an isomorphism $Z \cong Z'\times_{\spec \Fun} \spec k$, or in
more concrete terms, it means that there is an open affine covering of $Z$ for which the coordinate
rings are presented as monoid rings, and the localizations with which these affine patches are glued
are induced by localizations of the monoids.  See \cite[\S3]{GG1} and the references therein for
details.  This generalization of toric varieties allows, for instance, non-finite type schemes, and
more limits exist in this category.  In particular, while there is no initial object in the category
of embeddings of a scheme $X$ into toric varieties, the category of embeddings into schemes
equipped with an $\Fun$-model does have an initial object, which we denote $X\hookrightarrow
\widehat{X}$.  This universal embedding is completely explicit; on each affine patch it is the unit
transformation of the base-change adjunction between affine $\Fun$-schemes and affine $k$-schemes.
More concretely, if $X=\spec A$ is affine, $\widehat{X}$ is the spectrum of the monoid ring $k[A]$,
and $X \hookrightarrow \widehat{X}$ corresponds to the evaluation map $k[A] \twoheadrightarrow A$;
so the ``monomials'' are the elements of $A$.  This construction extends to the case when $X$ is not
affine.

In \cite{GG1} we showed how one can tropicalize a closed subscheme of any integral $k$-scheme equipped
with a model over $\Fun$.  In this paper we show that one can obtain the Berkovich analytification
not merely as a limit of tropicalizations, but as a tropicalization itself.

\begin{thmA}\label{thm:UnivBerk}
  Let $X$ be an integral scheme over a non-archimedean valued field $k$.  The
  set-theoretic tropicalization of $X$ with respect to the universal embedding $X \hookrightarrow
  \widehat{X}$ is canonically identified with the underlying set of the Berkovich analytification
  $X^{\an}$. Moreover, this identification becomes a homeomorphism when the tropicalization of $X$
  is endowed with the ``strong Zariski topology'' in which the closed subsets are given by closed
  subschemes.
\end{thmA}

\begin{remark}
It follows immediately from the universal property of the embedding $X\hookrightarrow
\widehat{X}$, and functoriality of tropicalization, that the corresponding tropicalization is the
limit of \emph{all} tropicalizations (not just those coming from embeddings in toric varieties).
The statement that one often obtains the same limit when restricting to a class of embeddings in
toric varieties \cite{Payne,Payne-inv2} requires an additional argument.
\end{remark}

\begin{remark}
Given a system of toric embeddings such that the corresponding limit of
tropicalizations is the analytification (cf., \cite{Payne-inv2}), one can obtain the analytification
as the tropicalization with respect to the (typically infinite) product of these toric embeddings.
However, not all varieties can be embedded in a toric variety \cite{Toric-emb},
so for such varieties one must leave the realm of toric tropicalization and embrace more general
$\Fun$-schemes in order to recover analytification as a tropicalization.
\end{remark}

The main construction introduced in \cite{GG1} is a scheme-theoretic refinement of tropicalization
that reduces to the set-theoretic tropicalization upon passing to the set of $\T$-points, where $\T$
is the tropical semiring $(\R\cup\{\infty\},\mathrm{min},+)$.  In brief, given a closed embedding
locally described as a quotient of a monoid ring $k[M]$, the tropicalization is a quotient of the
semiring $\T[M]$, and the equations of the tropicalization are produced by valuating the
coefficients of the original equations and then applying the \emph{bend relations} (see \S\ref{sec:tropreview} for a review). 

When we apply scheme-theoretic tropicalization to the universal embedding $X\hookrightarrow
\widehat{X}$ we obtain a semiring scheme underlying the analytification.

\begin{thmA}\label{thm:UnivBerk-scheme-theoretic}
  Let $X$ be an integral scheme over a non-archimedean valued ring $k$.  The
  scheme-theoretic tropicalization associated with the embedding $X \hookrightarrow \widehat{X}$
  admits a canonical morphism of $\T$-schemes to the tropicalization associated with any other
  closed embedding of $X$ into a locally integral scheme with $\Fun$-model.  Upon passing to $\T$-points,
  these morphisms reduce to the canonical projections from the Berkovich analytification to all
  set-theoretic tropicalizations.
\end{thmA} 

Accordingly, we call this scheme the \emph{universal tropicalization} of $X$ and denote it by
$\Trop_{univ}(X)$. 

Since $\Trop_{univ}(X)$ is the initial object in the category of tropicalizations (with morphisms
induced by morphisms of embeddings), it is trivially the limit (in the category of $\T$-schemes) of
the diagram of \emph{all} scheme-theoretic tropicalizations.  We also restrict to the subcategory of
tropicalizations from embeddings in affine space and prove a scheme-theoretic refinement of Payne's
affine limit result:

\begin{thmA}
  If $X$ is an affine integral scheme of finite type over $k$, then the limit of its tropicalizations with respect
  to all closed embeddings in finite-dimensional affine spaces is naturally isomorphic as a $\T$-scheme to the
  universal tropicalization $\Trop_{univ}(X)$.
\end{thmA}

\begin{remark}
For simplicity, we have restricted attention to the affine case of this theorem.  A modification of
the argument along the lines of \cite[Theorem 4.2]{Payne} should allow one to easily prove the
corresponding statement for embeddings of a quasiprojective variety $X$ into toric varieties.
\end{remark}

\subsection{The moduli space of semivaluations}

The category of $\T$-schemes and the universal tropicalization $\Trop_{univ}(X)$ allow us to make
precise the idea that the Berkovich analytification is a moduli space of semivaluations.  Let $A$ be a
$k$-algebra, where $k$ is a field equipped with a non-archimedean valuation $\nu: k \to \T$.  Recall
that the points of the analytification of $\spec A$ are the rank-one semivaluations on $A$ compatible with $\nu$.

In \cite[Definition 2.5.1]{GG1} and \cite[\S4.2]{Macpherson} a generalization of the notion of
semivaluation on a ring was introduced by replacing $\T$ with an arbitrary idempotent semiring $S$ --- a
semivaluation in this sense is now a multiplicative map $A \to S$ that satisfies a certain subadditivity
condition with respect to the canonical partial order on $S$.  If $k$ is equipped with a semivaluation
$\nu: k\to S$, and $T$ is an $S$-algebra, then a semivaluation $A\to T$ is said to be \emph{compatible}
with $\nu$ if the square
\[
\begin{diagram}
\node{k} \arrow{e} \arrow{s} \node{S} \arrow{s} \\
\node{A} \arrow{e} \node{T}
\end{diagram}
\]
commutes.

\begin{thmA}\label{thm:intro-moduli-of-semivaluations}
  Let $k$ be a field, $S$ an idempotent semiring, and $\nu: k \to S$ a semivaluation.  Given an integral
  $k$-algebra $A$, the $S$-scheme $\Trop_{univ}(\spec A)$ represents the functor on affine
  $S$-schemes sending $\spec T$ to the set of semivaluations $A\to T$ compatible with $\nu$.  In
  particular, there is a universal semivaluation on $A$ compatible with $\nu$ and it takes values in the
  semiring of regular functions on the universal tropicalization.
\end{thmA}

\begin{remark}
 This moduli functor of semivaluations was first shown to be representable in semiring schemes by
 MacPherson in \cite[Theorem 6.24]{Macpherson}.   Thus, the main novelty of this theorem is to
 identify the representing scheme with the universal tropicalization.
\end{remark}

Note that the generalized semivaluations we consider here include higher rank Krull (semi)valuations (where
$S$ is a totally ordered idempotent semifield).  Thus the universal tropicalization contains
information about Huber's adic space analytification \cite{Huber} in addition to the rank-one
information of the Berkovich analytification.  Note also that in order for the universal semivaluation
to exist, the total ordering on the value group for semivaluations must indeed be weakened to a partial
ordering, as in \cite[Definition 2.5.1]{GG1}.

\subsection*{Conventions}
Throughout this paper all algebraic objects will be assumed to be commutative.  Monoids, rings, and
semirings are always assumed to have a multiplicative unit, and semirings are always assumed to have
an additive unit as well.

\subsection*{Acknowledgements}
The first author was supported by EPSRC grant EP/I003908/2 and the second author by the NSF
postdoctoral fellowship DMS-1204440.  We thank Martin Ulirsch and Tommaso de Fernex for useful
comments and discussions. We indebted to Oliver Lorscheid for bringing to our attention a mistake in
the earlier version of this paper; in that version the extension of the universal embedding from
affine to non-affine schemes was not correct.  Finally, we thank the anonymous referee for
insightful comments that significantly improved parts of this paper.

\section{The universal embedding}

Let $X$ be a scheme over a ring $R$.  In this section we will construct an
embedding $X\hookrightarrow \widehat{X}$ that is universal among embeddings of $X$ that determine
tropicalizations of $X$. i.e., embeddings over $\spec R$ into schemes equipped with a locally
integral model over the field with one element $\Fun$.

\subsection{The evaluation map}

We shall use the naive version of the field with one element $\Fun$ as put forward by \cite{Kato,
  Deitmar1, Toen-Vaquie}.  Rather than define $\Fun$ as an object directly, one instead specifies
what its categories of modules and algebras should be.

\begin{definition}
  An \emph{$\Fun$-module} is a set equipped with a distinguished basepoint.  An
  \emph{$\Fun$-algebra} is a monoid-with-zero, i.e., a commutative monoid $B$ (written
  multiplicatively) with a multiplicative unit $1$ and an element $0$ such that $b \cdot 0 = 0$ for
  all $b\in B$.  An $\Fun$-algebra is said to be \emph{integral} if the set of nonzero elements
  forms a monoid (i.e., there are no zero-divisors) that is cancellative (i.e., the canonical map
  from this monoid to its group completion is injective).  A \emph{homomorphism} of $\Fun$-algebras
  is a monoid homomorphism sending $0$ to $0$.
\end{definition}

\begin{remark}
Note that an $\Fun$-algebra with no zero-divisors need not be integral.  For example, the
$\Fun$-algebra $B=\{0, 1, x, y, y^2, y^3, \cdots\}$ with the relation $xy=y^2$ has no zero-divisors,
but the monoid of non-zero elements fails to be cancellative.
\end{remark}

Let $R$ be a (semi)ring.  There is a forgetful functor 
\[
R\text{-mod} \to \Fun\text{-mod}
\]
that sends an $R$-module to its underlying set with $0$ as the distinguished point; we will refer to
this as a \emph{scalar restriction} functor.  It admits a left adjoint that sends an $\Fun$-module
to the free $R$-module generated by the non-basepoint elements; we call this the \emph{scalar
extension} functor and denote it by $-\otimes_\Fun R$.  This adjunction at the level of modules
induces an adjoint pair of functors
\begin{equation}\label{eq:adjunction}
-\otimes_\Fun R: \Fun\text{-alg} \leftrightarrows R\text{-alg} : M(-).
\end{equation}
The scalar restriction functor sends an $R$-algebra $A$ to its underlying multiplicative monoid
$M(A)$; we will write $x_a$ for the element corresponding of $a\in A$.  The scalar extension
functor sends an $\Fun$-algebra $B$ to the $R$-algebra with one generator $x_b$ for each element
$b\in B$ and the relations
\begin{align*}
x_a x_b &= x_{ab} \:\:\:\: \text{for $a,b\in B$},\\
x_1 &= 1,\\
x_0 &=0.
\end{align*}
Given an $R$-algebra $A$ and an $\Fun$-algebra $B$, the adjoint of an $R$-algebra homomorphism $f:
B\otimes_\Fun R \to A$ is the map $B\to M(A)$ sending $b$ to $x_{f(b)}$.

\begin{lemma}\label{lem:integral}
A ring $A$ is an integral domain if, and only if, the $\Fun$-algebra $M(A)$ is integral.
\end{lemma}

\begin{proof}
If $A$ is an integral domain, then it has no zero-divisors, so the set of non-zero elements $A
\smallsetminus \{0\}$ is a cancellative monoid (a non-cancellative relation $xy=xz$ would imply the
existence of a zero divisor: $x(y-z)=0$).  Conversely, if $A \smallsetminus \{0\}$ is a cancellative
monoid then $A \smallsetminus \{0\} = M(A) \smallsetminus 0$ injects into its group completion,
which is $M(\Frac(A))\smallsetminus\{0\}$, so $A$ injects into its field of fractions.
\end{proof}

Given an $R$-algebra $A$, we will write $\widehat{A}$ for the $R$-algebra $M(A) \otimes_\Fun R$
(i.e., first apply scalar restriction and then scalar extension). Elements of $\widehat{A}$ are
finite formal $R$-linear combinations of elements $x_a$ for $a\in A \smallsetminus \{0\}$. Note that
$\widehat{A}$ is functorial in $A$; an $R$-algebra homomorphism $f: A \to B$ induces an $R$-algebra
homomorphism $\widehat{f}: \widehat{A} \to \widehat{B}$ sending $x_a \mapsto x_{f(a)}$.  The
adjunction \eqref{eq:adjunction} comes with a counit natural transformation $\widehat{A} \to A$
which admits an explicit description.    The counit is a surjective $R$-algebra homomorphism  $ev$
defined by
\[ 
ev: x_a \mapsto a.
\]  
We call this map the \emph{evaluation} because it evaluates a
formal $R$-linear combination of elements of $A$ to an element of $A$ using the arithmetic
operations of $A$:
\[
ev \left( \sum_i \lambda_i x_{a_i} \right) = \sum_i \lambda_i a_i,
\]
for $\lambda_i \in R$ and $a_i \in A$.  

The kernel of $ev$ is
\[
\ker(ev) = \left\{ \sum_i \lambda_i x_{a_i}~ \Biggr\rvert ~\sum_i \lambda_ia_i = 0 \right\},
\] 
and it admits a much smaller presentation that will be useful later on.

\begin{proposition}\label{kernel-of-ev}
Given an $R$-algebra $A$, the kernel of $ev: \widehat{A} \to A$ is generated as a $\Z$-module by the
following elements:
\begin{enumerate}
\item $\lambda x_a - x_{\lambda a}$, for $a\in A$ and $\lambda \in R$;
\item $x_a + x_b + x_c$, for $a,b,c\in A$ with $a+b+c=0$.
\end{enumerate}

\end{proposition}
\begin{proof}
Given an expression $E = \sum^n_{i=1} \lambda_i x_{a_i}$ in the kernel of $ev$, we will reduce it to zero by
adding an appropriate sequence of elements of the two types above.  Adding $x_{\lambda_i a_i} -
\lambda_i x_{a_i}$ to $E$, for each $i$, reduces it to the expression 
\[
E'= \sum^n_{i=1} x_{\lambda_i a_i}.
\]
This is in the kernel of $ev$ if and only if $\sum_i \lambda_i a_i = 0$.  Subtracting first
$x_{\lambda_1 a_1} + x_{\lambda_2 a_2} + x_{-\lambda_1 a_1 - \lambda_2 a_2}$ and then
$-x_{-\lambda_1 a_1 - \lambda_2 a_2} - x_{\lambda_1 a_1 + \lambda_2 a_2}$ from $E'$ yields an
expression of the same form but with one fewer terms.  Repeating this inductively eventually yields
$x_{\sum_i \lambda_i a_i}$; the subscript is zero by the hypothesis that the original expression $E$
was in the kernel of $ev$, and $x_a = 0$ in $\widehat{A}$ if and only if $a=0$.
\end{proof}

\subsection{The universal embedding in the affine case}\label{sec:affine-embedding}

We now describe how the evaluation map $ev: \widehat{A} \to A$ gives a universal embedding of $\spec
A$ into an ambient space with the structure required to tropicalize $\spec A$. 

\begin{definition}
An \emph{$\Fun$-framing} on an $R$-scheme $X$ is an $\Fun$-scheme $Y$ and a morphism of $R$-schemes $X \to
Y \times_\Fun \spec R$.  We say that a framing is affine if $X$ and $Y$ are both affine.  A morphism
of $\Fun$-framed schemes is a diagram
\begin{equation}\label{eq:framing-morphism-diagram}
\begin{diagram}
\node{X} \arrow{e} \arrow{s} \node{Y \times_\Fun \spec R} \arrow{s}\\
\node{X'} \arrow{e} \node{Y' \times_\Fun \spec R}
\end{diagram}
\end{equation}
where the right vertical arrow is induced by a morphism of $\Fun$-schemes  $Y \to Y'$.  An
\emph{$\Fun$-embedding} is an $\Fun$-framing that is also a closed immersion.
\end{definition}
In other words, the category of $\Fun$-framed schemes is simply the comma category of
$R$-schemes over $\Fun$-schemes.

For any affine $R$-scheme $X = \spec A$, the evaluation map defines an $\Fun$-embedding  $X
\hookrightarrow \widehat{X}$ where the $\Fun$-algebra of the ambient space is $M(A)$.  We call this
the \emph{universal embedding} of $X$ for the following reason.  As a purely formal consequence of the
adjunction \eqref{eq:adjunction}, this embedding is the initial object in the category
of framings on $X$, and hence it is initial among $\Fun$-embeddings.  More precisely,

\begin{proposition}\label{prop:initial-embedding}
Given an affine framed $R$-scheme $\alpha: X \to \spec B \otimes_\Fun R$, there is a unique
factorization through the universal embedding
\[
  X \to \widehat{X} \stackrel{\eta_\alpha} \to \spec B \otimes_\Fun R.
\]
Moreover, this factorization is functorial in the sense that a morphism of affine framed schemes as
\eqref{eq:framing-morphism-diagram} induces a diagram
\[
\begin{diagram}
\node{X} \arrow{s} \arrow{e} \node{\widehat{X}} \arrow{e} \arrow{s} \node{Y \times_\Fun \spec R} \arrow{s}\\
\node{X'} \arrow{e} \node{\widehat{X'}} \arrow{e} \node{Y' \times_\Fun \spec R.}
\end{diagram}
\]
\end{proposition}
\begin{proof}
Suppose $X=\spec A$, so the framing $\alpha$ corresponds to an $R$-algebra homomorphism 
\[
B \otimes_\Fun R \to A,
\]
 which is adjoint to an $\Fun$-algebra homomorphism
\begin{equation}\label{eq:fun-morphism-for-univ-property}
B \to M(A).
\end{equation}  
The evaluation map $\widehat{A} \to A$ is adjoint to the identity on $M(A)$.  The desired map
$\eta_\alpha: \widehat{X} \to \spec B\otimes_\Fun R$ corresponds to a homomorphism
\[
  B\otimes_\Fun R \to M(A)\otimes_\Fun R
\]
that is the functor $-\otimes_\Fun R$ applied to a morphism $B \to M(A)$, but the only such morphism
that gives the desired factorization is \eqref{eq:fun-morphism-for-univ-property}.  This shows
existence and uniqueness of the factorization $\eta_\alpha$.  The functoriality claim follows from a
similar purely formal manipulation of the adjunction.
\end{proof}

\subsection{The difficulty in globalizing the universal embedding}\label{sec:globadjoint}

We would like to construct an extension of the universal embedding functor from affine $R$-schemes
to arbitrary schemes over $\spec R$, and we will do so in the next section.  However, this
is not entirely straightforward.  The root of the difficulty is that $X \mapsto \widehat{X}$ does
not in general send open covers to open covers, as we now explain. 

\begin{remark}
In the earlier version of this paper, it was incorrectly claimed that the adjunction between scalar
restriction and scalar extension for affine schemes directly extends to an adjunction for arbitrary
schemes. Although open immersions are sent to open immersions, since the adjunction does not
preserve covers, existence of the extension is not immediate.  We thank Oliver Lorscheid for
pointing this out.
\end{remark}

Recall that the category of affine $\Fun$-schemes is the opposite of the category of
$\Fun$-algebras, basic open immersions correspond to localizations, and general $\Fun$-schemes can be
described by gluing affine patches together along open immersions.  See \cite[\S3]{GG1} or
\cite{Toen-Vaquie} for further details.  The category of schemes over a semiring can also be
constructed in essentially the same way.

The scalar restriction and extension functors $ R\text{-alg} \leftrightarrows \Fun\text{-alg}$
commute with localizations (see, for example, \cite[Paragraph 6.1.13]{Durov}), so they send open
immersions to open immersions. However, the scalar restriction functor $M$ does not send open covers
to open covers. The difficulty here is that an affine $\Fun$-scheme has very few Zariski open
covers. 

\begin{proposition}
Given an affine $\Fun$-scheme $Y$, a collection of affine open subschemes $\{U_i\}$ is a Zariski
cover if and only if $U_i=Y$ for some $i$.
\end{proposition}
\begin{proof}
Suppose $Y=\spec B$ for an $\Fun$-algebra $B$.  A collection of principal open subschemes $U_i = \spec
B[a_i^{-1}]$ covers $\spec B$ if and only if the $a_i$ generate the unit ideal in $B$.  The ideal
generated by the $a_i$ is simply the set of all elements that are a multiple of one of the $a_i$;
i.e., it is the multiplicative monoid generated by the $a_i$.  This is the unit ideal if and only if
at least one of the $a_i$ is a unit.  
\end{proof}

\begin{example}
Let $A = R[z]$ and consider the principal open subschemes of $X=\spec A$ given by $U=\spec
R[z,z^{-1}]$ and $V=\spec R[z, (1-z)^{-1}]$.  The pair $\{U,V\}$ is an open cover since $z$ and
$1-z$ generate the unit ideal in the ring $A$.  The scalar restriction functor $M$ sends $U$ and $V$
to principal open subschemes of the affine $\Fun$-scheme $\spec M(A)$. The ideal $(x_{z},
x_{1-z})\subset M(A)$  consists of all elements that can be written in either the form $x_a x_z$ or
$x_a x_{1-z}$ for some $a\in A$.  Clearly this is not the unit ideal of $M(A)$ since $1 \neq az$ or
$a(1-z)$ for any $a\in A$.  On $\widehat{X} = \spec M(A) \otimes_\Fun R$, we have open subschemes
$\widehat{U}$ and $\widehat{V}$ defined by the ideals $(x_{z})$ and $(x_{1-z})$ respectively, but
these ideals do not generate the unit ideal in the $R$-algebra $M(A) \otimes_\Fun R$ so these open
subschemes don't cover $\widehat{X}$.
\end{example}

Thus if $X$ is an affine $R$-scheme with an affine open cover $\{U_i\}$, the schemes $\widehat{U_i}$
will be affine open subschemes of $\widehat{X}$, but they will not in general form an open cover.
The problem, then, is that if $X$ is not necessarily affine we cannot build $\widehat{X}$ simply by
choosing an affine open cover and gluing the universal embeddings of the affine patches in this
cover. Different choices of cover will result in different glued objects.  In the next subsection,
we will see that the way around this obstacle is to use \emph{all} open affine subschemes of $X$.

\subsection{The universal embedding in the non-affine case}

We now extend the universal embedding to arbitrary not necessarily affine schemes. The idea is that
a scheme $X$ can be written as the colimit of \emph{all} its affine open subschemes $U$, and so we
define $\widehat{X}$ to be the colimit of the corresponding diagram of affine schemes $\widehat{U}$
and open immersions.  This colimit exists (one can construct it as a locally ringed space and then
easily check that it is a scheme). 

\begin{proposition}\label{prop:globalfactor}
The universal embeddings $U \to \widehat{U}$ for affine open subschemes of $X$ glue together to give
a embedding $X \to \widehat{X}$, and on affine schemes this restricts to the previously constructed
affine universal embedding. This embedding is initial among $\Fun$-framings of $X$.
\end{proposition}
\begin{proof}
For each affine open $U$, there is a canonical framing morphism $\alpha_U: U \to \widehat{U} \to \widehat{X}$,
and when $U \subset V$, the functoriality of Prop. \ref{prop:initial-embedding} implies that
$\alpha_V|_U = \alpha_U$, so these morphisms glue together to give a closed immersion $X \to
\widehat{X}$.  

To show that $X \to \widehat{X}$ is initial, suppose $Y$ is an $\Fun$-scheme and $\alpha: X \to
Y\times_\Fun \spec R$ is a morphism.  The target is covered by affine patches of the form $\spec
B\otimes_\Fun R$, and for any affine open $U \subset X$ mapping into $\spec
B\otimes_\Fun R$, there is a unique factorization 
\[
U \to \widehat{U} \to \spec B\otimes_\Fun R.
\]
Once again, by the functoriality of these affine factorization, these glue together to produce the
desired factorization $X \to \widehat{X} \to Y\times_\Fun \spec R$.  Uniqueness of this
factorization follows from the fact that it is locally unique.
\end{proof}

The gluing construction above can also be described in category theoretic terms as a left Kan
extension from affine schemes to schemes. First, let us recall the concept of left Kan extension.
Consider a pair of functors
\[
\begin{diagram}
\node{\mathcal{A}} \arrow{e,t}{F} \arrow{s,l}{G} \node{\mathcal{B}} \\
\node{\mathcal{C}} \arrow{ne,..}
\end{diagram}
\]
and the question of filling in the diagonal arrow.  There may or may not exist a functor $H:
\mathcal{C} \to \mathcal{B}$ making the diagram commute in the sense that there is a natural
isomorphism $F \cong H \circ G$.  A weaker condition would be to ask for an $H$ and a natural
transformation (that is not necessarily invertible) $\omega_H: F
\Rightarrow H \circ G$, but now there is a category of such pairs $(H,\omega_H)$.  A morphism
$(H_1, \omega_1) \to (H_2, \omega_2)$ is a natural transformation $H_1 \Rightarrow H_2$ that sends
$\omega_1$ to $\omega_2$.  The \emph{left Kan extension of $F$ along $G$} is an initial object in
this category; i.e., it is a functor $\Lan_G F: \mathcal{C} \to \mathcal{B}$ together with a natural
transformation $\omega: F \Rightarrow \Lan_G F \circ G$ such that for any other pair $(H,\omega_H)$
there is a unique natural transformation $\Lan_G F \Rightarrow H$ sending $\omega$ to $\omega_H$.
While it may not always exist, when $\mathcal{B}$ admits enough colimits, the left Kan extension can
be constructed via a straightforward recipe:
\[
  \Lan_G F (c) = \colim_{a \in G / c} F(a)
\]
where $G/c$ is the comma category of objects over $a\in \mathcal{A}$ over $c$, i.e., pairs $(a\in
\mathcal{A}, f: G(a) \to c)$.

Returning to the case at hand, let $\mathcal{O}$ denote the category of $R$-schemes and open
immersions, and let $\mathcal{O}_{\mathit{Aff}} \subset \mathcal{O}$ denote the full subcategory of
affine schemes. The affine universal embedding $X \mapsto
\widehat{X}$ constructed so far in section \ref{sec:affine-embedding} is a functor defined on affine
schemes and \emph{all} morphisms, but let us now restrict to open immersions and regard it as a functor
\[
  E: \mathcal{O}_\mathit{Aff} \to \mathrm{Sch}/R.
\]
Consider the left Kan extension of $E$ along the inclusion $j$:
\[
\begin{diagram}
\node{\mathcal{O}_\mathit{Aff}} \arrow{e,t}{E} \arrow{s,l}{j} \node{\mathrm{Sch}/R} \\
\node{\mathcal{O}} \arrow{ne,b}{\Lan_j E}
\end{diagram}
\]
Somewhat more concretely, for an $R$-scheme $X$, the object $\widehat{X} := \Lan_j E (X)$
can be described as the colimit of $\widehat{U}$ as $U$ runs over the poset
$\mathcal{O}_\mathit{Aff}(X)$ of all affine open subschemes of $X$ (this is the category of objects
of $\mathcal{O}_\aff$ over $X \in \mathcal{O}$).

\section{Tropicalizing the universal embedding}

We now study the tropicalization of an integral scheme $X$ with respect to the universal
embedding $X\hookrightarrow \widehat{X}$.

\subsection{A brief review of scheme-theoretic tropicalization}\label{sec:tropreview}

In \cite{GG1} we introduced a generalization and refinement of the Kajiwara-Payne set-theoretic
tropicalization of subvarieties of toric varieties over a rank-one valued field.  Here we review that
construction.

Let $\T$ denote the idempotent semiring $(\R\cup\{\infty\},\mathrm{min},+)$ with additive unit
$0_\T = \infty$ and multiplicative unit $1_\T=0$.  Let $k$ be a field equipped with a
valuation $\nu: k \to \T$ (by which we mean a multiplicative and subadditive map
preserving the multiplicative and additive unit, respectively).  Let $B$ be an integral
$\Fun$-algebra, and $I\subset B\otimes_\Fun k$ an ideal.  We can regard $I$ as a $k$-linear subspace
and tropicalize it with respect to $\nu$ to get a tropical linear space $\trop(I)\subset
B\otimes_\Fun \T$ (or $\trop^\nu(I)$ if we need to emphasize the valuation) which is, by definition,
the $\T$-linear span of the coefficient-wise valuations of the elements of $I$. The set $\trop(I)$
is automatically a $\T$-submodule, and moreover it turns out to be an ideal in $B\otimes_\Fun \T$
(this requires the assumption that $B$ is integral, see \cite[Proposition 6.1.1 and Remark
6.1.2]{GG1}); i.e., it is a \emph{tropial ideal} as studied by Maclagan and Rin\'con in
\cite{Maclagan-Rincon-1,Maclagan-Rincon-2,Maclagan-Rincon-3}. The congruence $\bend\trop(I)$ on
$B\otimes_\Fun \T$ is generated by the \emph{bend relations}
\[
\bend(f): \:\: f \sim f_{\widehat{b}}
\]
for $f\in \trop(I)$ and $b$ a monomial term in $f$, where $f_{\widehat{b}}$ denotes the result of
deleting $b$ from $f$.

Now let $X$ be a $k$-scheme, $Y$ a locally integral $\Fun$-scheme (which means $Y$ admits an open
affine cover by the spectra of integral $\Fun$-algebras), and $\varphi: X\hookrightarrow
Y\times_\Fun k$ a closed embedding corresponding to a quasi-coherent ideal sheaf $\mathscr{I}$ on $Y
\times_\Fun k$.  For each integral affine patch $U\subset Y$, $\mathscr{I}$ is in particular a
$k$-linear subspace of the space of regular functions on $U\times_\Fun k$, and so its
tropicalization $\trop(\mathscr{I}(U))$ is a tropical linear space in the space of regular functions
on $U\times_\Fun\T$. These tropical linear spaces assemble to form a quasi-coherent ideal sheaf
$\trop(\mathscr{I})$ on $Y\times_\Fun \T$.  Applying the bend relations $\bend(-)$ on each of the
above affine patches then yields a quasi-coherent congruence sheaf $\bend\trop(\mathscr{I})$ on
$Y\times_\Fun\T$; the tropicalization $\Trop_\varphi(X)$ (or $\Trop_\varphi^\nu(X)$ to emphasize the
valuation) of $X$ with respect to the embedding $\varphi$ is then, by definition, the closed
subscheme determined by this congruence sheaf, regarded as a scheme over $\spec \T$.

A quintessential example of an $\Fun$-scheme is a toric variety $Y_\Delta$, where the fan $\Delta$
provides the model over $\Fun$.  Set-theoretic tropicalization, as defined by Payne and Kajiwara
\cite{Payne, Kajiwara}, applies to subvarieties of toric varieties $\varphi : X \hookrightarrow
Y_\Delta$, and \cite[Theorem 6.3.1]{GG1} shows that the output of that coincides with the
$\T$-points of the scheme-theoretic tropicalization $\Trop_\varphi(X)$.

\begin{remark}
The idea of generalizing the ambient spaces for tropicalization from toric varieties to
$\Fun$-schemes first appeared in \cite{Popescu-Pampu-Stepanov}, although they worked only with
set-theoretic tropicalization.
\end{remark}

Functoriality of tropicalization \cite[Proposition 6.4.1]{GG1} is a scheme-theoretic enrichment of
Payne's observation \cite[Corollary 2.6]{Payne-fibers} regarding torus equivariant morphisms of
toric varieties.  Namely, if $\psi : Y_1 \rightarrow Y_2$ is a map of integral $\Fun$-schemes and
$\varphi_i : X \hookrightarrow Y_i\times_\Fun k$ are closed embeddings forming a commutative
triangle
\[
\begin{diagram}
\node[2]{X} \arrow{sw,t}{\varphi_1} \arrow{se,t}{\varphi_2} \\
\node{Y_1 \times_\Fun k} \arrow[2]{e,b}{\psi\times_\Fun k} \node[2]{Y_2 \times_\Fun k,}
\end{diagram}
\]
then there is an induced morphism of tropicalizations $\Trop_{\varphi_1}(X) \rightarrow
\Trop_{\varphi_2}(X)$.

In addition to passing from toric varieties to arbitrary integral $\Fun$-schemes as ambient spaces for
tropicalization, in \cite{GG1} we observed that the domain of scheme-theoretic tropicalization, with
its functoriality property, naturally admits the following enlargement:
\begin{enumerate}
\item The field $k$ can be replaced by an arbitrary ring $R$.
\item The tropical numbers $\T$ can be replaced by an arbitrary idempotent semiring $S$.
\end{enumerate}
Item (2) requires a generalization of the definition of (semi)valuation that we discuss in the next
section. This generalization includes the case of higher-rank Krull (semi)valuations $k \rightarrow
\Gamma \cup \{-\infty\}$ simply by giving the totally ordered abelian group $\Gamma$ the structure
of a semiring with $-\infty$ as the additive identity, where multiplication is the group operation
in $\Gamma$ and addition is the maximum with respect to the ordering.  Set-theoretic
tropicalizations with respect to higher rank Krull valuations were studied in \cite{Banerjee}.
Higher rank tropicalization has been further studied in \cite{Foster-Ranganathan}, where an
associated notion of higher rank `Hahn' analytification is introduced.

The output of $\Trop$ is then canonically a scheme over $\spec S$. Explicitly, when $Y = \spec B
\otimes_\Fun R$ and $X$ is defined by an ideal $I$, we have
\[
  \Trop(X) = \spec B\otimes_\Fun S / \bend \trop(I) 
\]
where $\trop(I) \subset B\otimes_\Fun S$ is the ideal generated by applying a (generalized)
semivaluation $R\to S$ coefficient-wise to the elements of $I$.

\subsection{Semivaluations}
Here we review the definition and properties of the generalized class of semivaluations that were first
introduced in \cite{GG1}.

\begin{definition}
Given a ring $R$ and an idempotent semiring $S$, a semivaluation $\nu: R \to S$ is a map such that
\begin{enumerate}
\item $\nu(0)=0_S$;
\item $\nu(1)=\nu(-1)=1_S$;
\item (multiplicative): $\nu(ab)=\nu(a)\nu(b)$;
\item (subadditive): $\nu(a+b) +  \nu(a) + \nu(b) = \nu(a)+\nu(b)$.
\end{enumerate}
\end{definition}
Note that conditions (2) and (3) allow us to rewrite the subadditivity condition (4) in a more symmetric form as 
\[
  \nu(a) + \nu(b) + \nu(c) = \nu(a) + \nu(b) = \nu(a) + \nu(c) = \nu(b) + \nu(c)
\] 
for any $a,b,c \in R$ satisfying $a+b+c=0$.   More generally, it follows that if $a_1 + a_2 + \cdots
+ a_n = 0$ in $R$, then the value of the summation $\nu(a_1) + \cdots + \nu(a_n)$ in $S$ is unchanged if any
one term is omitted; we say that this summation \emph{tropically vanishes}. Also observe that if $R
\to S$ is a semivaluation in this sense, then pre-composition with a ring homomorphism $R' \to R$ and
post-composition with a semiring homomorphism $S \to S'$ both yield semivaluations.

Note also that when $S=\mathbb{T}$, the above definite reduces to the usual definition of a real (rank 1) semivaluation. 

Let $\val(R,S)$ denote the set of semivaluations from $R$ to $S$.  Fixing $S$, we can regard $\val(-,S)$
as a covariant functor $\mathcal{O}_\mathit{aff} \to \mathrm{Sets}$.  We can then take the left Kan
extension along the inclusion $\mathcal{O}_\mathit{aff} \hookrightarrow \mathcal{O}$ to extend to
arbitrary schemes.   Concretely, a point of $\val(X,S)$ is represented by an affine open
subscheme $U = \spec A$ of $X$ and a semivaluation $\nu: A\to S$, and two pairs $(U_1,\nu_1)$ and $(U_2,
\nu_2)$ are equivalent if there exists $(U_3, \nu_3)$ with $U_3 \subset U_1 \cap U_2$ such that
$\nu_3$ maps to $\nu_1$ and $\nu_2$. We refer to such an object as a semivaluation on $X$ with values in
$S$.  This mildly extends the definition of \cite[\S3.1]{Temkin} to allow semivaluations taking valued
in more general semirings instead of just $\mathbb{T}$.

If $\nu: R\to S$ is a semivaluation, $A$ is an $R$-algebra, and $T$ is an $S$-algebra, then we say that a semivaluation
$w: A\to T$ is \emph{compatible} with $\nu$ if the diagram
\[
\begin{diagram}
\node{R} \arrow{s} \arrow{e} \node{S} \arrow{s} \\
\node{A} \arrow{e} \node{T}
\end{diagram}
\]
commutes.   We write $\val_\nu(A,T)$ for the set of semivaluations on $A$ taking values in $T$ and
compatible with $\nu$; it is a functor of $R$-algebras in the first variable, and of $S$-algebras in
the second variable.  As in the paragraph above, left Kan extension along the inclusion
$\mathcal{O}_\mathit{aff} \hookrightarrow \mathcal{O}$ extends this to a functor on arbitrary
schemes.

When the semivaluation on $R$ takes values in $\mathbb{T}$ and $X$ is an $R$-scheme, the
subset $\val_\nu(X,\mathbb{T})$ of semivaluations compatible with $\nu$ is precisely the underlying set of the
Berkovich analytification $X^{\an}$.

\subsection{The points of a tropicalization}
Let $Y$ be an $\Fun$-scheme, and let $R$ be a ring with a semivaluation $\nu: R\to S$.  Given an
$R$-algebra $A$, an $S$-algebra $T$, and a semivaluation $w: A \to T$ compatible with $\nu$, there
is a \emph{tropicalization-of-points} map
\[
  \trop: (Y\times_\Fun \spec R) (A) \to (Y \times_\Fun \spec S) (T).
\]
This map is defined as follows.  Locally $Y\times_\Fun \spec R)$ is of the form $\spec B
\otimes_\Fun R$ for an $\Fun$-algebra $B$, and an $A$-valued point is adjoint to a homomorphism $B
\to M(A)$.  Since $w$ is multiplicative, the composition \[B \to M(A) \stackrel{w}{\to} M(T)\] is an $\Fun$-algebra
homomorphism, and this is adjoint to a $S$-algebra homomorphism $B\otimes_\Fun S \to T$.

Let $X$ be a locally integral scheme over $R$ with semivaluation $\nu: R \to S$, and $\varphi: X
\hookrightarrow Y \times_\Fun \spec R$ an $\Fun$-embedding.  These data determine a tropicalization
$\Trop_\varphi(X)$, which is a scheme over the idempotent semiring $S$.  Locally, $X$ is cut out by
an ideal $I \subset B\otimes_\Fun R$, and
\[
  \Trop_\varphi(X) = \spec B\otimes_\Fun S / \bend \trop(I) 
\]
where $\trop(I) \subset B\otimes_\Fun S$ is the ideal generated by applying $\nu$ coefficient-wise
to the elements of $I$.  Note that $X(A) \subset (Y\times_\Fun R) (A)$ and $\Trop_\varphi(X)(T)
\subset (Y\times_\Fun S) (T)$.

\begin{proposition}
The tropicalization-of-points map sends $X(A)$ into $\Trop_\varphi(X)(T)$.
\end{proposition}
\begin{proof}
It suffices to verify this on an affine patch $\spec B\otimes_\Fun R \subset Y\times_\Fun R$. If $X$
is defined by an ideal $I$, then an $A$-valued point $p$ corresponds to an $R$-algebra morphism $p:
B\otimes_\Fun R \to A$ such that $f(p) = 0$ for all $f \in I$.  Thus $w(f(p)) = 0_T$, and then it
follows from the symmetric form of the subadditivity condition that if $f_i$ are the monomial terms
of $f$, then the sum $\sum_i w(f_i(p))$ is unchanged if any single term is omitted.  It follows that
$p$ tropicalizes to a homomorphism  $\trop(p): B\otimes_\Fun S \to T$ that descends to the quotient
by the congruence $\bend\trop(I)$, and hence $w\circ p$ determines a $T$-valued point of
$Trop_\varphi(X)$.
\end{proof}

\begin{remark}
A $T$-valued point of $\Trop_\varphi(X)$ can be thought of as a semivaluation on $X$ that is, in an
affine patch on which it lives, defined only on the monomials coming from the embedding $\varphi$
and satisfying subadditivity only for linear combinations of monomials in the defining ideal of $X$.
Moreover, every such partially defined semivaluation comes from a $T$-point.  This bijection appears as
Theorem $D$ in \cite{F1-vector-bundles}, and it follows directly from the definitions here.  This
fact provides an illuminating interpretation of our results on the relation between the universal
tropicalization or Berkovich analytification and finite tropicalizations:  a semivaluation on a scheme
$X$ restricts to a partially defined semivaluation on the monomials from any $\Fun$-embedding of $X$,
and if the set of monomials is all elements of the coordinate algebra of an affine patch in $X$ then
a partially defined semivaluation is the same as a true semivaluation.
\end{remark}

\begin{remark}
The tropicalization-of-points map is a mild generaization of the Kajiwara-Payne tropicalization map
for toric varieties \cite{Kajiwara,Payne}.  There are numerous interesting further refinements and
generalization of tropicalization, such as the local tropicalization of Popescu-Pampu and Stepanov
\cite{Popescu-Pampu-Stepanov} and the logarithmic tropicalization studied by Ulirsch
\cite{Ulirsch1} and its extension to stacks \cite{Ulirsch2}. It would be interesting to combine our
construction with these.
\end{remark}

\subsection{Strong tropical bases}
Let $B$ be an integral $\Fun$-algebra, $\nu : R \rightarrow S$ a semivaluation from a ring to an
idempotent semiring, and $I\subset B\otimes_\Fun R$ an ideal.  In general, the tropicalized ideal
$\trop(I)\subset B\otimes_\Fun S$ is not finitely generated, and so the congruence $\bend\trop(I)$
is presented by an infinite set of generating relations.  There is often a large amount of
redundancy in this generating set, and so one can ask about the existence of smaller sets of
generating relations. Given a subset $K \subset I$, we can consider the set of tropical polynomials
$\{ \nu(f) \}_{f\in K}$ and then the congruence generated by their bend relations, $\langle \bend(\nu(f)) \rangle_{f\in K}$.

\begin{definition}
  A \emph{strong tropical basis} for an ideal $I$ is a generating subset $K\subset I$
   such that the
\[
  \bend\trop(I) = \langle \bend(\nu(f)) \rangle_{f\in K}.
\]
\end{definition}

\begin{remark}
  In \cite[Definition 8.2.1]{GG1} we proposed a different notion of tropical basis: a
  scheme-theoretic tropical basis is a collection $\{J_i\}$ of principal ideals that generates $I$
  as an ideal and such that the collection of congruences $\{\bend\trop(J_i)\}$ generates
  $\bend\trop(I)$.  If $\{f_i\}\subset I$ is a strong tropical basis then the collection of
  principal ideals $\{(f_i)\}$ is always a scheme-theoretic tropical basis.  However, the converse
  is not necessarily true: as shown in \cite[Example 8.1.1]{GG1}, for the polynomial
  $f=x^2+xy+y^2 \in R[x,y]$, the congruence $\bend(\nu(f))$ is strictly smaller than the congruence
  $\bend\trop(I)$ for $I=(f)$, and so in this case $\{(f)\}$ is trivially a scheme-theoretic
  tropical basis, while $\{f\}$ is \emph{not} a strong tropical basis. 
\end{remark}

The following result is the technical heart of this paper.  

\begin{proposition}\label{prop:strongtropbas}
The elements given in Proposition \ref{kernel-of-ev} are a strong tropical basis for the kernel of
$ev: \widehat{A} \to A$, for any semivaluation $\nu : R \rightarrow S$.
\end{proposition}
\begin{proof}
  By \cite[Lemma 5.1.3(2)]{GG1} (see also the first paragraph in the proof of \cite[Theorem
  1.1]{Maclagan-Rincon-1}), the congruence $\bend\trop(\ker(ev))$ is spanned as an $S$-module by the
  congruences $\bend(\nu(f))$ for $f\in \ker(ev)$.  Thus it suffices to show that each such
  congruence $\bend(\nu(f))$ is contained in the congruence generated by the bend relations of the
  semivaluations of the elements of types (1) and (2).

  Consider an element $f\in \ker(ev)$; the coefficient-wise semivaluation of $f$ is of the form
  \[\nu(f)=\sum_i \nu(\lambda_i) x_{a_i},\text{ with }\sum_i \lambda_i a_i = 0.\]  The bend relations of the
  semivaluations of generators of type (1) give the relations
  \begin{equation}\label{eq:type1}
    x_{\lambda a} \sim \nu(\lambda)x_a, 
  \end{equation}
  and in particular, $x_{-a} \sim x_a$ since $\nu(-1)=\nu(1) = 1$.  Using these relations we see
  that $\nu(f)$ is equivalent to an element $g\in M(A)\otimes_\Fun S$ of the form
  \[
  g = \sum^n_{i=1} x_{b_i}, \text{ with }\sum b_i = 0
  \]
  (here $b_i := \lambda_i a_i$ to keep the notation simpler).  We now show that the congruence
  $\bend(g)$ is contained in the congruence
  \[
  J :=\langle\bend(x_a + x_b + x_{-a-b})\rangle_{a,b\in A}.
  \]
  The relation $x_{b_1} + x_{b_n} \sim x_{b_2 + \cdots + b_{n-1}} + x_{b_n}$ from
  $\bend(x_{b_1} + x_{b_2 + \cdots + b_{n-1}} + x_{b_n})$ gives the relation
  \begin{align}\label{eq:g-equiv1}
    g  &= x_{b_1} + x_{b_n} + \sum_{i=2}^{n-1} x_{b_i} \\
    & \sim x_{b_2 + \cdots + b_{n-1}} + x_{b_n} +  \sum_{i=2}^{n-1} x_{b_i} \nonumber \\
    & = g_{\widehat{b_1}} + x_{b_2 + \cdots + b_{n-1}} \nonumber
  \end{align}
  in $J$.  Next, consider the relation
  \[
  x_{b_2 + \cdots + b_\ell} + x_{b_{\ell+1}} \sim x_{b_2 + \cdots + b_{\ell+1}} + x_{b_{\ell+1}}
  \]
  from $\bend(x_{b_2 + \cdots + b_\ell} + x_{b_{\ell+1}} + x_{-b_2 - \cdots - b_{\ell+1}})$ and $x_c
  \sim x_{-c}$; using this repeatedly as $\ell$ runs from 2 up to $n-2$ gives
  \begin{align*}\label{eq:g-equiv2}
    g_{\widehat{b_1}} & = x_{b_2} + x_{b_3} + \cdots + x_{b_n} \\
    & \sim x_{b_2 + b_3} + x_{b_3} + \cdots + x_{b_n} \\
    &\quad\vdots \\
    & \sim x_{b_2 + \cdots + b_{n-1}} + x_{b_3} + \cdots + x_{b_n}
  \end{align*}
  in $J$, and hence the idempotency of addition implies that the relation 
  \[
  g_{\widehat{b_1}}+ x_{b_2 + \cdots + b_{n-1}} \sim g_{\widehat{b_1}}
  \]
  is in $J$.  Combining this with \eqref{eq:g-equiv1} yields the desired relation $g\sim
  g_{\widehat{b_1}}$ in $J$.  Since the choice of ordering of the $b_i$ was arbitrary, this shows
  that the bend relations of $g$ are indeed all contained in $J$.  Using the bend relations of the
  type (1) elements once again, but this time in the reverse of the direction we used them when
  passing from $\nu(f)$ to $g$, we have the relation $g_{\widehat{b_i}} \sim \nu(f)_{\widehat{a_i}}$
  in $J$.  Combined with the bend relations of $g$, this shows that the bend relations $\nu(f)$ are
  entirely contained in $J$.
\end{proof}

\subsection{The universal tropicalization}

Let $S$ be an idempotent semiring, $\nu: R\to S$ a semivaluation, and $X$ a scheme over $R$.
Note that the universal embedding $X\hookrightarrow \widehat{X}$ defined in
\S\ref{sec:globadjoint} yields a tropicalization when the $\Fun$-model $M(X)$ of $\widehat{X}$ is
locally integral.  When $X$ is irreducible, this is equivalent to $X$ being integral; indeed, this
can be checked on sufficiently small affine patches by \cite[Proposition 3.1.3]{GG1}, and it holds
there by Lemma \ref{lem:integral}.   We therefore assume in this section that $X$ is integral and we
will study its tropicalization in $\widehat{X}$.

\begin{definition}
  The \emph{universal tropicalization of $X$}, denoted $\Trop^\nu_{univ}(X)$, is the tropicalization
  of $X$ with respect to the canonical closed embedding $X\hookrightarrow \widehat{X}$.
\end{definition}

\begin{proposition}\label{prop:univ-property-of-univ-trop}
  Let $Y$ be an integral $\Fun$-scheme and $\varphi: X\hookrightarrow Y\times_\Fun R$ a closed
  embedding.  There is a canonical morphism of $S$-schemes $\Trop_{univ}(X) \to \Trop_{\varphi}(X)$,
  and it is natural in both $X$ and $(Y,\varphi)$.
\end{proposition}
\begin{proof}
This follows immediately from the functoriality of tropicalization \cite[Proposition 6.4.1]{GG1} and
the universal property of the embedding $X\hookrightarrow \widehat{X}$ described in Proposition
\ref{prop:globalfactor}.
\end{proof}
This proves the first part of Theorem \ref{thm:UnivBerk-scheme-theoretic}, and it is because of the
above result that the tropicalization of $X$ in $\widehat{X}$ deserves to be called the universal
tropicalization of $X$.

\begin{theorem}\label{thm:moduli-of-semivaluations}
  The universal tropicalization
  $\Trop^\nu_{univ}(X)$ represents the contravariant functor on affine $S$-schemes sending $\spec T$
  to the set $\val_\nu(X,T)$ of semivaluations $X \rightarrow T$ compatible with $\nu$.
\end{theorem}

Thus the universal tropicalization is the algebraic moduli space of semivaluations on $X$.  When $R$ is
a field $k$, we obtain Theorem \ref{thm:intro-moduli-of-semivaluations} from the introduction.
Moreover, when $S$ is the tropical semiring $\T$ we obtain the set-theoretic bijection part of
Theorem \ref{thm:UnivBerk} directly from this by passing to the set of $\T$-points, since on the one
hand these are the points of the set-theoretic tropicalization, and on the other hand these are the
semivaluations $X \rightarrow \T$, i.e., the points of the Berkovich analytification; the part of the
theorem describing the Berkovich topology is explained and proven in \S\ref{sec:topology} below.

\begin{proof}[Proof of of Theorem \ref{thm:moduli-of-semivaluations}]
  It suffices to assume that $X=\spec A$ is affine.  In this case the universal tropicalization is
  \[
  \spec M(A)\otimes_\Fun S / \bend\trop(\ker (ev)).
  \]
  By Proposition \ref{prop:strongtropbas}, a $T$-point of this is a multiplicative map $\alpha: A
  \to T$ such that
  \begin{equation}\label{eq:val-compatible}
  \alpha(\lambda a) = \nu(\lambda)\alpha(a) \text{ for $\lambda\in R$ and $a\in A$,}
  \end{equation}
  and (using the fact that $\alpha(-c)=\nu(-1)\alpha(c) = \alpha(c)$ from the equation above),
  \begin{align}\label{eq:val-bend-rels}
    \alpha(a) + \alpha(b) + \alpha(a+b) & = \alpha(a) + \alpha(b)\\
    & = \alpha(a) + \alpha(a+b) \nonumber\\
    &= \alpha(b) + \alpha(a+b). \nonumber
  \end{align}
  The first condition \eqref{eq:val-compatible} says that $\alpha$ is compatible with the
  semivaluation on $k$.  In the second condition, \eqref{eq:val-bend-rels}, the first equality is
  precisely the subadditivity condition for a semivaluation. We now observe that the remaining two
  equalities are actually redundant and so impose no additional conditions.  We have $\alpha(a) +
  \alpha(a+b) = \alpha(a) + \alpha(-(a+b))$, and by the first equality of \eqref{eq:val-bend-rels}
  (applied with $a$ and $-(a+b)$ instead of $a$ and $b$), this is equal to \[\alpha(a) +
  \alpha(-(a+b)) + \alpha(a-(a+b))=\alpha(a) + \alpha(a+b) +
  \alpha(b).\]  Thus the second equality of \eqref{eq:val-bend-rels} follows from the first, and by
  symmetry between $a$ and $b$ the third one does as well.
\end{proof}

Note that if $X=\spec A$ is an affine $R$-scheme, then 
\[
\Trop^\nu_{univ}(\spec A) = \spec M(A)\otimes_\Fun S / \bend\trop(\ker ev)\] is an affine
$S$-scheme.  Let us abbreviate its algebra of global functions $S_A$.
\begin{corollary}\label{cor:universal-semivaluation}
There is semivaluation $w: A \to S_A$ that is universal among all semivaluations compatible with $\nu$ in the
following sense:  given any semivaluation $w': A\to T$ compatible with $\nu$, there is a unique homomorphism
of $S$-algebras $f: S_A \to  T$ such that $w'=f\circ w$.
\end{corollary}
The universal semivaluation sends $a\in A$ to $x_a \in \spec
M(A)\otimes_\Fun S / \bend\trop(\ker ev)$.

\subsection{An aside on the relation between universal valuations and universal tropicalization}

The universal tropicalization discussed above is defined using a fixed semivaluation $\nu$ on the
base $R$ and the universal embedding of $X$, in contrast to the tropicalization
$\Trop_\varphi^{\nu^R_\mathit{{univ}}}(X)$ from \cite[\S6.5]{GG1} which is defined with a fixed
embedding and the universal semivaluation $\nu^R_{univ}$ on $R$.  The point of Theorem
\ref{thm:intro-moduli-of-semivaluations} (Theorem \ref{thm:moduli-of-semivaluations} and Corollary
\ref{cor:universal-semivaluation}) is that these two objects are closely related, as we now explain.

Given a ring $A$, one can consider the category of all semivaluations on $A$, and this category
contains an initial object given by the universal semivaluation $\nu^A_{\mathit{univ}}$.  If $A$ is
an $R$-algebra, then one can also consider the subcategory of all semivaluations compatible with a
given semivaluation $\nu$ on $R$, and this subcategory contains an initial object taking values in
the coordinate algebra of the universal tropicalization (i.e., tropicalization of the universal
embedding).

Note that every semivaluation on $A$ is compatible with the universal semivaluation
$\nu^R_\mathit{univ}$ on $R$.  Hence, if we consider the universal tropicalization of $\spec A$ with
respect to the universal semivaluation $\nu^R_{univ}$ on $R$, the coordinate algebra of the
resulting object will be the semiring in which the universal semivaluation on $A$ takes values.

In geometric terms, the space of all semivaluations on $R$ is (an enrichment of) the Berkovich
spectrum $\mathcal{M}(R)$, and the universal tropicalization of $\spec A$ with respect to the
universal semivaluation on $R$ can be viewed as the family of Berkovich analytifications of $\spec
A$ parametrized by $\mathcal{M}(R)$.

\subsection{The Berkovich topology as an algebraic topology}\label{sec:topology}
We now introduce the \emph{strong Zariski topology} on a semiring scheme, and in the case of a
universal tropicalization we show that this topology coincides with the Berkovich topology.

Let $X$ be a scheme over $S$, viewed as a topological space with a structure sheaf. The Zariski
topology on the underlying set of $X$ induces a Zariski topology on the $S$-points $X(S)$.
Explicitly, an open immersion $U \hookrightarrow X$ determines a subset $U(S)\subset X(S)$, and
every Zariski open subset is of this form.

\begin{definition}
Let $X$ be an $S$-scheme.  The \emph{strong Zariski topology} on $X(S)$ is the topology whose closed
subsets are of the form $Z(S)$ for $Z$ a closed subscheme of $X$.
\end{definition}

Note that since the pullback of a closed subscheme is again a closed subscheme, a morphism of
schemes $X\to X'$ induces a map $X(S) \to X'(S)$ that is continuous with respect to the strong Zariski topology.

\begin{proposition}
  If $S$ is a ring then the strong Zariski topology and the ordinary Zariski topology coincide, but
  in general the strong Zariski topology is finer than the Zariski topology.
\end{proposition}
\begin{proof}
  The strong Zariski closed subsets are defined by equations of the form $f = g$, whereas Zariski closed
  subsets are defined by equations of the restricted form $f = 0_S$.  I.e., Zariski closed subsets
  are given by ideals, while strong Zariski closed subsets are given by congruences.  Over a ring
  congruences and ideas are in bijection, and so it follows that the strong Zariski topology and the Zariski
  topology on $X(S)$ coincide.
\end{proof}

\begin{remark}
  When $S$ is not a ring then the two topologies can be distinct.  For example, over $\N$ the
  diagonal in $\mathbb{A}^1\times \mathbb{A}^1$ is strong Zariski closed but not Zariski closed
  (cf. \cite[\S6.5.19]{Durov}), and for $\mathbb{A}^1(\N) = \mathbb{N}$ the strong Zariski topology is
  the finite complement topology, while the only nontrivial Zariski closed subset is the singleton
  $\{0\}$.
\end{remark}

Note that the structure sheaf $\mathscr{O}_X$ is a sheaf with respect to the Zariski topology, but
\emph{not} with respect to the strong Zariski topology.

The \emph{Euclidean topology} on $\T$ is the topology for which the exponential map gives a
homeomorphism with $\mathbb{R}_{\geq 0}$. More generally, if $\Lambda$ is a (possibly infinite) set then
the Euclidean topology on the product space $\T^\Lambda$ is given by the product topology with the
Euclidean topology on each factor.

\begin{lemma}\label{lem:infinite-product-topology}
  Consider the affine space $\mathbb{A}^\Lambda_\T = \spec \T[x_i \:\:|\:\: i\in \Lambda]$.  The
  strong Zariski topology on $\mathbb{A}^\Lambda(\T) = \T^\Lambda$ is exactly the Euclidean topology.
\end{lemma}
\begin{proof}
  We first show that strong Zariski closed sets are also Euclidean closed.  A strong Zariski closed set $Z$
  is a (possibly infinite) intersection of principal strong Zariski closed subsets
  $V(f\sim g) = \{x\in \T^\Lambda \:\: | \:\: f(x) = g(x)\}$.  Since a polynomial $f\in \T[x_i
  \:\:|\:\: i\in \Lambda]$ has only finitely many terms corresponding to a finite subset $\supp f
  \subset \N^\Lambda$, its graph is a finite type polyhedron $\Gamma_f$ in $\T^{\supp f}\times \T$ crossed with
  $\T^{\N^\Lambda \smallsetminus \supp f}$.  The polyhedra $\Gamma_f$ and $\Gamma_g$ are both Euclidean
  closed, so their intersection is, and hence the set $V(f\sim g)$ is Euclidean closed.  Therefore
  strong Zariski closed sets are Euclidean closed.

  We now show that there is a basis for the Euclidean topology consisting of strong Zariski open sets.  A
  basis of open sets for the Euclidean topology is given by Euclidean open boxes:
  $\prod_{i\in \Lambda} J_i$ with each $J_i$ an open interval in $\T$, and all but finitely many of
  them are the whole of $\T$.  Each $J_i= (a_i, b_i)$ or $[-\infty,b_i)$ is strong Zariski open, as it is
  the complement of the strong Zariski closed set $V(x_i +a_i \sim x_i) \cap V(x_i + b_i \sim b_i) $ or
  $V(x_i +b_i \sim b_i)$, respectively. 
\end{proof}

\begin{remark}
As far as we are aware, this observation that closed subschemes generate the Euclidean topology goes
back originally to Mikhalkin in \cite[Prop. 2.22]{Mikhalkin}.
\end{remark}

\begin{theorem}
The strong Zariski topology on $\Trop_{univ}(X)$ coincides with the Berkovich topology.
\end{theorem}
\begin{proof}
Since $\Trop_{univ}(X)$ has an open cover by subschemes of the form $\Trop_{univ}(U)$ for $U$ an
affine open subscheme of $X$, it suffices to assume that $X$ is affine.  The claim follows directly
from Lemma \ref{lem:infinite-product-topology} and the fact that the Berkovich topology on $X^{an}$
coincides with the subspace topology it gets from the natural inclusion into $\T^{k[X]}$.
\end{proof}

\begin{example}
Let $k$ be an algebraically closed field that is complete with respect to a non-archimedean absolute
value $|-|$.  Consider $(\mathbb{A}^1_k)^{\mathit{an}}$;  the points of this space are
semivaluations on $k[t]$ that are compatible with the absolute value on $k$.  Given $a\in k$ and
$r\in \mathbb{R}_{>0}$, a Berkovich open disk is a subset or the form
\[
  D_\mathit{an}(a,r) = \{\nu \:\: | \:\: e^{-\nu(t-a)} < r \}.
\]
A fundamental system of neighborhoods for the Berkovich topology is given by the finite
intersections of these Berkovich open disks.  The complement of an open disk $D_\mathit{an}(a,r)^c$
is the strong Zariski closed set defined by the equation
\[
  x_{t-a} = x_{t-a} + (-\log r)x_1.
\]
Note that $x_1$ is the multiplicative unit.
\end{example}

\subsection{The map from the analytification to a tropicalization}

Let $Y$ be a locally integral $\Fun$-scheme and $X$ a $k$-scheme, with a closed embedding $\varphi: X
\hookrightarrow Y\times_\Fun k$.  Given a rank-one valuation $\nu : k \rightarrow \T$, there is a
canonical map $\pi$ from the Berkovich analytification to the set-theoretic tropicalization of $X$
with respect to $\varphi$,
\[
\pi: X^{\an} \to \trop_\varphi(X).
\]
This is a slight generalization of the map constructed by Payne in \cite{Payne}.  It can be
described on a suitable affine patch as follows. Suppose $X$ is given by $\spec A$ for some
$k$-algebra $A$, the $\Fun$-scheme $Y$ is $\spec B$ for some $\Fun$-algebra $B$, and the embedding
$\varphi$ is given by a surjective homomorphism $\varphi^\sharp: B\otimes_\Fun k \twoheadrightarrow
A$.  A point of the analytification is a semivaluation $w: A\to \T$ compatible with $\nu$, and a
point of the set-theoretic tropicalization is a $\T$-algebra homomorphism $q: B\otimes_\Fun \T \to
\T$ such that for each $\sum \lambda_i x_{b_i} \in\ker\varphi^\sharp$, the maximum of the set
$\{q(\nu(\lambda_i) x_{b_i})\}$ is either equal to $-\infty$ or is attained at least twice.  Given a point
$w\in (\spec A)^{\an}$, the composition
\[
B \hookrightarrow B\otimes_\Fun k \stackrel{\varphi^\sharp}{\to} A \stackrel{w}{\to} \T
\]
is multiplicative and so determines a $\T$-algebra homomorphism $\pi(w): B\otimes_\Fun \T \to \T$.

\begin{proposition}
  The homomorphism $\pi(w)$ lies in $\trop_\varphi(X)$.
\end{proposition}
\begin{proof}
  Let $f=\sum \lambda_i x_{b_i}$ be an element in the kernel of $\varphi^\sharp$, so $w \circ \varphi^\sharp (\sum \lambda_i
  x_{b_i}) = -\infty$.  On the other hand, taking the coefficient-wise valuation, $\nu(f) = \sum
  \nu(\lambda_i)x_{b_i}$, and so
\begin{align*}
\pi(w) (\nu(f)) & = \sum \nu(\lambda_i) w\circ \varphi^\sharp (x_{b_i}) \\
                      & = \sum w\circ\varphi^\sharp (\lambda_i x_{b_i}).
\end{align*}
Thus we have 
\begin{equation}\label{eq:wphisharp-ineq}
-\infty = w \circ \varphi^\sharp (\sum \lambda_i x_{b_i}) \leq \sum w\circ\varphi^\sharp (\lambda_i x_{b_i}).
\end{equation}
Since $w\circ \varphi^\sharp: B\otimes_\Fun k \to \T$ is a semivaluation and $\T$ is totally ordered, a strict inequality
\[w\circ \varphi^\sharp (a+b) < w\circ \varphi^\sharp (a) + w\circ \varphi^\sharp (b)\] implies $w\circ \varphi^\sharp (a) = w\circ \varphi^\sharp (b)$, and
so we can conclude from the above inequality \eqref{eq:wphisharp-ineq} that the maximum of $\{w\circ\varphi^\sharp (\lambda_i
x_{b_i})\} \subset \T$ occurs at least twice (or there is only a single term and it is $-\infty$).  This
shows that $\pi(w)$ is indeed in the set-theoretic tropicalization $\trop_\varphi(X)$.
\end{proof}

\begin{proposition}
Upon passing to $\T$-points, the canonical map of $\T$-schemes \[\Trop_{univ}(X) \to
\Trop_\varphi(X)\] reduces to the map $\pi: X^{\an} \to \trop_\varphi(X)$.
\end{proposition}
\begin{proof}
  This is just a mater of unwinding the definitions.  Proposition \ref{prop:globalfactor} provides the canonical map $\widehat{X} \to Y\times_\Fun k$
  for which we have a commuting diagram,
\[
\begin{diagram}
\node[2]{X} \arrow{sw} \arrow{se,t}{\varphi} \\
\node{\widehat{X}} \arrow[2]{e,t}{\widehat{\pi}} \node[2]{Y\times_\Fun k.}
\end{diagram}
\]
By restricting attention to suitable affine patches, this diagram is represented at the level of
$k$-algebras by a diagram
\[
\begin{diagram}
\node[2]{A} \\
\node{M(A)\otimes_\Fun k} \arrow{ne,t}{ev} \node[2]{B\otimes_\Fun k,} \arrow[2]{w,t}{\widehat{\pi}^\sharp} \arrow{nw,t}{\varphi^\sharp}
\end{diagram}
\]
for a $k$-algebra $A$ and an $\Fun$-algebra $B$.  The bottom arrow $\widehat{\pi}^\sharp$ in this diagram is induced
by the morphism of $\Fun$-algebras
\[
\varphi^\flat : B\to M(A)
\]
 that is adjoint to
$\varphi^\sharp$; i.e., $\varphi^\flat$ sends $x_b\in B\otimes_\Fun k$ to $\varphi^\sharp(x_b)$ (thought of as an element
of $M(A)$).  The associated morphism of tropicalizations,
\[
B\otimes_\Fun \T / \bend\trop(\ker \varphi^\sharp) \to M(A)\otimes_\Fun \T / \bend\trop(\ker ev)
\]
is also induced by $\varphi^\flat$.  Consider a $\T$-point \[w: M(A)\otimes_\Fun \T /
\bend\trop(\ker ev) \to \T\] of the universal tropicalization; it is entirely determined by the
morphism of $\Fun$-algebras $w^\flat: M(A) \to M(\T)$ given by restricting $w$ to monomials.  The
map $\Trop_{univ}(X) \rightarrow \Trop_\varphi(X)$ sends $w$ to the $\T$-point corresponding to the composition of $\Fun$-algebra morphisms 
\[B\stackrel{\varphi^\flat}{\to} M(A) \stackrel{w^\flat}{\to} M(\T),\]
and one easily sees that this agrees with the description of $\pi(w)$ we gave above.
\end{proof}

By assembling Proposition \ref{prop:univ-property-of-univ-trop}, Theorem
\ref{thm:moduli-of-semivaluations}, and the above proposition, we have proven Theorem
\ref{thm:UnivBerk-scheme-theoretic} from the introduction.

\section{Limits of tropicalizations}

Fix an integral scheme $X$ over a valued ring $\nu :R \rightarrow S$, and let $\mathcal{C}$
denote the category of `locally integral $\Fun$-embeddings of $X$'; that is, an object of
$\mathcal{C}$ is a locally integral $\Fun$-scheme $Y$ together with a closed embedding $X
\hookrightarrow Y\times_\Fun R$, and a morphism is a morphism of $\Fun$-schemes inducing a
commutative triangle of $R$-schemes.  By \cite[Proposition 6.4.1]{GG1}, scheme-theoretic
tropicalization yields a covariant functor $\Trop_\bullet(X) : \mathcal{C} \rightarrow \Sch_S$.

Because of its universal property (Proposition \ref{prop:globalfactor}), the universal embedding
$X\hookrightarrow \widehat{X}$ is an initial object in $\mathcal{C}$.  Thus we trivially have that
the universal tropicalization is the limit over $\mathcal{C}$ of all tropicalizations of $X$.

\begin{proposition}
There is a canonical isomorphism, 
\[\lim_{\varphi\in\mathcal{C}} \Trop_\varphi(X) \cong \Trop_{univ}(X).\]
\end{proposition}

It is more interesting to consider the limit over certain subcategories of $\mathcal{C}$, such as
the subcategory $\mathcal{C}_{\mathbb{A}}$ of all embeddings into affine spaces  $\mathbb{A}^n$ (for
varying finite $n$) and torus-equivariant morphisms. Payne showed that if $X$ is an affine variety,
then the limit over $\mathcal{C}_{\mathbb{A}}$ of the set-theoretic tropicalizations of $X$ (considered as
topological spaces) is homeomorphic to $X^{\an}$ \cite[Theorem 1.1]{Payne}.  In the following
section we give a scheme-theoretic refinement of this theorem.

\subsection{Affine embeddings}

Let $X=\spec A$, for $A$ a finitely generated $R$-algebra.  The category $\mathcal{C}^{\op}_{\mathbb{A}}$
admits the following explicit algebraic description.
\begin{itemize}
\item[Objects:] Finitely generated free $\Fun$-algebras (i.e., finite rank free abelian monoids-with-zero) $B$ equipped with a surjective $R$-algebra homomorphism $B \otimes_\Fun R \twoheadrightarrow A$; equivalently, there is a specified $\Fun$-algebra homomorphism $B \rightarrow M(A)$ whose image generates $A$ as an $R$-algebra.
\item[Arrows:] Homomorphisms of $\Fun$-algebras $B_1 \rightarrow B_2$ whose scalar extension commutes with the maps $B_i \otimes_\Fun R \twoheadrightarrow A$; equivalently, these are $\Fun$-algebra homomorphisms over $M(A)$.
\end{itemize}

Functoriality of tropicalization therefore gives in this case a functor $\mathcal{C}_{\mathbb{A}}
\rightarrow \Sch_S$, or $\mathcal{C}^{\op}_{\mathbb{A}} \rightarrow S\text{-alg}$.

\begin{theorem}\label{thm:affine-inv-limit}
  Let $A$ be a finitely generated integral $R$-algebra, and suppose $R$ is equipped with a semivaluation
  $\nu: R\to S$. The universal tropicalization of $X = \spec A$ is isomorphic as an $S$-scheme to the
  limit of tropicalizations in affine spaces:
\[\Trop^\nu_{univ}(X) \cong \lim_{\varphi\in \mathcal{C}_{\mathbb{A}}} \Trop^\nu_\varphi(X).\]
\end{theorem}

We begin with a lemma, which is used to show that the limit is embedded in the affine $S$-scheme $\spec M(A)\otimes_\Fun S$.  It says that the limit of the affine spaces in which $X$ embeds is $\widehat{X}$.  Let $F: \mathcal{C}^{\op}_{\mathbb{A}} \to \Fun\text{-alg}$ be the forgetful functor sending $B\otimes_\Fun R \twoheadrightarrow A$ to $B$.

\begin{lemma}\label{lem:colim}
$\colim_{\mathcal{C}^{\op}_{\mathbb{A}}} F \cong M(A)$. 
\end{lemma}
\begin{proof}
  First, note that the colimit exists because the category of abelian monoids is cocomplete and the
  colimit of a diagram of monoids-with-zero clearly has a zero element and is the colimit in the
  subcategory of monoids-with-zero.  Let $Z$ denote this colimit.  Since arrows in $\mathcal{C}^{\op}_{\mathbb{A}}$ are
  $\Fun$-algebra morphisms over $M(A)$, the universal property of the colimit yields a canonical morphism
  of $\Fun$-algebras $Z\to M(A)$.  We show that this map is surjective and injective.

  To show that each $a\in M(A)$ is in the image of $Z \rightarrow M(A)$, it suffices to show that
  there is an object $B \otimes_\Fun R \twoheadrightarrow A$ of $\mathcal{C}^{\op}_{\mathbb{A}}$ whose restriction $B
  \rightarrow M(A)$ contains $a$ in its image.  Any finite set $\mathcal{S}\subset A$ of $R$-algebra
  generators containing $a$ yields a surjection $\Fun[x_1,\ldots,x_{|\mathcal{S}|}]\otimes_\Fun R =
  R[x_1,\ldots,x_{|\mathcal{S}|}] \twoheadrightarrow A$ with the desired property.
 
  Now we prove injectivity. Suppose $a\in A$ and $a_1, a_2 \in Z$ are two elements that both map to
  $a$.  Each $a_i$ can be represented by an element $a'_i$ in some finitely generated free
  $\Fun$-algebra $\Fun[\mathcal{S}_i]$ over $M(A)$, and without loss of generality we can assume
  $a_i'\in \mathcal{S}_i$.  Let $\mathcal{T} := \mathcal{S}_1 \cup_{a_1'\sim a_2'}\mathcal{S}_2$.
  We have a set-map $\mathcal{T} \rightarrow A$ induced by the maps $\mathcal{S}_i \rightarrow A$,
  since $a_1'$ and $a_2'$ have the same image in $A$.  These maps $\mathcal{S}_i \rightarrow A$
  factor through the inclusions $\mathcal{S}_i \hookrightarrow \mathcal{T}$, so the image of
  $\mathcal{T}$ in $A$ generates $A$ as a $R$-algebra, and hence $\Fun[\mathcal{T}] \rightarrow
  M(A)$ is an object of $\mathcal{C}^{\op}_{\mathbb{A}}$.  Moreover, the inclusions $\mathcal{S}_i \hookrightarrow
  \mathcal{T}$ induce arrows in $\mathcal{C}^{\op}_{\mathbb{A}}$ under which $a_1'$ and $a_2'$ are identified.  Applying the
  functor $F$ to these two arrows yields a pair of $\Fun$-algebra morphisms $\Fun[\mathcal{S}_i]
  \rightarrow \Fun[\mathcal{T}]$ for which the images of $a'_1$ with $a'_2$ coincide, and hence the
  images of these elements in the colimit $Z$ must be identified as well; i.e., $a_1= a_2$ in $Z$.
\end{proof}

We now turn to the proof of Theorem \ref{thm:affine-inv-limit}.
\begin{proof}
  Let $G : \mathcal{C}^{\op}_{\mathbb{A}} \rightarrow S$-alg be the functor sending an affine embedding to the algebra of
  global sections of the structure sheaf of the corresponding tropicalization:
\[
(B \otimes_\Fun R \stackrel{\psi}{\twoheadrightarrow} A) \mapsto
 B\otimes_\Fun S / \bend\trop(\ker \psi).
\] 
We must show that $V:= \colim_{\mathcal{C}^{\op}_{\mathbb{A}}}G$ is isomorphic to $W:= M(A) \otimes_\Fun S
/ \bend\trop(\ker ev)$. Functoriality of tropicalization applied to Proposition
\ref{prop:globalfactor} yields, by the universal property of the colimit, a canonical map of
$S$-algebras $V\to W$.  Since colimits commute with tensor products and quotients, it follows from
Lemma \ref{lem:colim} that $V$ is a quotient of $M(A)\otimes_\Fun S$, say by a congruence $J$.
Then, since the map \[V = M(A)\otimes_\Fun S/J \rightarrow M(A) \otimes_\Fun S /
\bend\trop(\ker ev) = W\] is induced by the identity on $M(A)$, we
see that it is surjective and $J \subset \bend \trop(\ker ev)$.  We will show that
this inclusion on congruences is an equality, i.e., that every relation in the universal tropicalization appears
at some stage of the colimit diagram.
  
By Proposition \ref{prop:strongtropbas}, it suffices to consider the bend relations coming from the
two types of basis elements of the kernel of the evaluation map described in Proposition
\ref{kernel-of-ev}.  Let $a,b,c\in A$ satisfy $a+b+c=0$, and consider the congruence
$\bend(x_a+x_b+x_c)$ on $M(A)\otimes_\Fun S$.  Choose a set $\mathcal{S} \subset A$ of $R$-algebra
generators containing $a,b,c$, corresponding to an object $\psi : \Fun[\mathcal{S}]\otimes_\Fun R
\twoheadrightarrow A$ of $\mathcal{C}^{\op}_{\mathbb{A}}$.  Then $x_a + x_b +x_c \in \ker\psi$, so the
congruence $\bend \trop(\ker \psi)$, which defines the tropicalization $G(\psi)$, contains the bend
relations of $x_a + x_b + x_c \in \Fun[\mathcal{S}]\otimes_\Fun S$.  Since the isomorphism in Lemma
\ref{lem:colim} is induced by the structure maps in the objects of $\mathcal{C}^{\op}_{\mathbb{A}}$, the
map $G(\psi) \rightarrow V$ sends this polynomial to $x_a+x_b+x_c$ in $M(A)\otimes_\Fun S$.  Thus
the bend relations $\bend(x_a+x_b+x_c)$ are contained in the colimit congruence $J$.  Similarly, for
$\lambda\in R$ and $a\in A$, an $R$-algebra generating set $\mathcal{T} \subset A$ containing $a$
and $\lambda a$ yields an object $\gamma : \Fun[\mathcal{T}]\otimes_\Fun R \twoheadrightarrow A$ of
$\mathcal{C}^{\op}_{\mathbb{A}}$ satisfying $\lambda x_a + x_{\lambda a} \in \ker\gamma$.  In the
congruence defining the tropicalization $G(\gamma)$ we thus have the bend relation $\nu(\lambda)x_a
\sim x_{\lambda a}$, and hence this also holds in the colimit congruence $J$.
\end{proof}

\bibliographystyle{amsalpha}
\bibliography{bib}

\end{document}